\documentclass{amsart}

\usepackage[utf8]{inputenc}
\usepackage[T1]{fontenc}
\usepackage{textcomp}
\usepackage{amsfonts}
\usepackage{amsthm}
\usepackage{amssymb}
\usepackage{pst-node}
\usepackage{ bbold }
\usepackage[all]{xy}
\usepackage{mathtools}
\usepackage{chngcntr}
\usepackage{thmtools}
\usepackage{mathrsfs}
\usepackage{tcolorbox}
\usepackage{listings}
\usepackage{imakeidx}
\usepackage{hyperref}
\usepackage{color}
\usepackage{xcolor}
\usepackage{float}
\usepackage{tikz}
\usepackage{enumerate}
\usepackage[nameinlink, capitalize, noabbrev]{cleveref}
\usepackage{natbib}
\usetikzlibrary{fit}

\hypersetup{
    colorlinks=true,
    linkcolor=teal,
    citecolor=magenta,
    }

\newtheorem{Theorem}{Theorem}

\newtheorem{Corollary}[Theorem]{Corollary}
\newtheorem{proposition}{Proposition}[section]
\newtheorem{lemma}[proposition]{Lemma}

\newtheorem{theorem}[proposition]{Theorem}

\newtheorem{Question}[Theorem]{Question}
\theoremstyle{definition}

\def \<#1>{{\left\langle{#1}\right\rangle}}
\def\abs#1{\left\vert{#1}\right\vert}
\def\set#1{{\def\st{\;:\;}\left\{#1\right\}}}

\def\NN{\mathbb N}
\def\ZZ{\mathbb Z}
\def\Z-{\overline{\mathbb Z}}
\def\Zl-{\overline{\mathbb Z}_\ell}

\def\Q-{\overline{\mathbb Q}}
\def\Ql-{\overline{{\mathbb Q}_\ell}}
\def\K-{\overline{K}}

\def\Fl-{\overline{{\mathbb F}_\ell}}

\DeclareMathOperator{\St}{St}
\DeclareMathOperator{\RiSt}{RiSt}

\DeclareMathOperator{\Aut}{Aut}

\DeclareMathOperator{\Sym}{Sym}

\numberwithin{equation}{section}

\title{Torsion elements in branch pro-$p$ groups}
\author{Jorge Fariña-Asategui and Santiago Radi}
\address{Jorge Fariña-Asategui: Centre for Mathematical Sciences, Lund University, 223 62 Lund, Sweden -- Department of Mathematics, University of the Basque Country UPV/EHU, 48080 Bilbao, Spain}
\email{jorge.farina\_asategui@math.lu.se}
\address{Santiago Radi: Department of Mathematics, Texas A\&M University, 77843 Co-llege Station, U.S.A.
}
\email{santiradi@tamu.edu}
\thanks{The first author is supported by the Spanish Government, grant PID2020-117281GB-I00, partly with FEDER funds. The first author also acknowledges support from the Walter Gyllenberg Foundation from the Royal Physiographic Society of Lund. The second author is supported by Grigorchuk's Simons Foundation Grant MP-TSM-00002045 and the department of Mathematics of Texas A\&M University.}

\begin{document}

\begin{abstract}
We show that the set of torsion elements of a topological group admitting a branch pro-$p$ quotient has Haar measure zero.
\end{abstract}

\maketitle

\section{introduction}
\label{section: introduction}

In 1902, Burnside proposed the following question known as the general Burnside problem: if a group is finitely generated and torsion, must the group be finite? In the same article, Burnside raised the same question for groups with uniformly bounded exponent, known as the bounded Burnside problem. Finally, in 1940, the restricted Burnside problem was formulated in \cite{Grun}: Given $m, n \in \NN$, are there finitely many finite groups with $m$ generators and exponent $n$?

The general Burnside problem was answered negatively by Golod in 1964, who constructed a counterexample by using the Golod-Shafarevich inequality; see either \cite{GolodShafarevich} or \cite[Theorem 7.21]{Koch2002}. The bounded Burnside problem was also answered negatively by Adian and Novikov in \cite{NovAdi68}. However, the restricted Burnside problem was solved positively by Zelmanov in \cite{RestBurnsideproblem}.

The counterpart of the general Burnside problem for compact groups was asked later by Platonov in the Kourovka Notebook \cite{Kourovka}. Using the classification of finite simple groups, Wilson reduced Platonov's question to the pro-$p$ case in \cite{WilsonProfinite}. Finally the pro-$p$ case was solved by Zelmanov in \cite{ZelmanovPro-p} who obtained the following positive result:

\begin{theorem}[see {\cite[Theorem 1]{ZelmanovPro-p}}]
\label{theorem: Zelmanov}
Every torsion pro-$p$ group is locally finite. 
\end{theorem}

In other words, \cref{theorem: Zelmanov} says that an infinite topologically finitely generated pro-$p$ group $G$ cannot be torsion. Still, a pro-$p$ group is compact, it has a unique normalized Haar measure, and we wonder whether the set of torsion elements can be large, i.e. whether it can have positive measure. In this direction, our main result is the following:

\begin{Theorem}
\label{Theorem: torsion}
Let $G$ be a branch pro-$p$ group and $\mu$ the Haar measure on $G$. Let $F$ be the set of torsion elements in~$G$. Then
$$\mu(F)=0.$$
\end{Theorem}

As a corollary of \cref{Theorem: torsion}, we obtain:

\begin{Corollary}
\label{Corollary: onto branch groups}
Let $G$ be a pro-$p$ group admitting a branch pro-$p$ quotient. If $F$ is the set of torsion elements of $G$ and $\mu$ the Haar measure on $G$, then $$\mu(F) = 0.$$
\end{Corollary}

Branch groups were introduced by Grigorchuk in 1997 at the St. Andrew's group theory conference in Bath and first appeared in print in \cite{Grigorchuk2000}. Branch groups arise in several areas of mathematics such as automata theory, fractal geometry, dynamical systems, topology or probability; see \cite{Self_similar_groups}. Fruitful connections with Galois theory, via the Fontaine-Mazur Conjecture and arboreal representations, arose as soon as branch groups were defined, and they are still an active area of research recently discussed during the workshop "Groups of dynamical origin" in 2024 in Pasadena \cite{AIM}; see \cite{NewHorizonsBoston, BostonJones} and the work on Boston's conjecture by the first author in \cite{JorgeSpectra}. An extended discussion on branch groups can be found in \cite{bartholdi2005branch}. 

There have been several results about branch groups in the last two decades, including studies of maximal subgroups \cite{Anitha, Bondarenko, Dominik, DominikAlejandra, Pervova_Maximal_subgroup}, the block structure of finitely generated subgroups \cite{francoeur2024structure}, and subgroups of finite index \cite{GarridoWilson}, among others. 

Recall that an infinite abstract group is just-infinite if every proper quotient is finite and an infinite profinite group is just-infinite if every proper quotient by a closed normal subgroup is finite. In 2000, Grigorchuk observed in \cite[Theorem 3, page 146]{Grigorchuk2000}, based on the results in \cite{Wilson1972}, that just-infinite branch groups constitute one of the three classes of just-infinite groups in the case of abstract groups and one of the two disjoint classes in the case of profinite just-infinite groups. On the other hand, it is observed in \cite[Proposition 3]{Grigorchuk2000} that every topologically finitely generated pro-$p$ group can be mapped onto a just-infinite pro-$p$ group, so \cref{Corollary: onto branch groups} applies to a wide family of topologically finitely generated pro-$p$ groups.

In \cref{sec_Example} we provide an example of a topologically finitely generated pro-$p$ group where the measure of its set of torsion elements is non-zero. This group has a normal subgroup of finite index that is isomorphic to a finite number of copies of the $p$-adic integers $\ZZ_p$, and thus it is $p$-adic analytic. This example shows that \cref{Theorem: torsion} does not hold in the class of $p$-adic analytic pro-$p$ groups. This leads us to post the following question:

\begin{Question}
If a topologically finitely generated infinite pro-$p$ group does not admit an infinite $p$-adic analytic quotient, must the Haar measure of its torsion elements be zero?
\end{Question}

\subsection*{Remark}

When the result was obtained, we were informed that the result stated in \cref{Theorem: torsion} follows from the result of Abért in \cite[Corollary 1.4]{abert2003group} and in fact his result works for a larger class of groups known as weakly branch groups, that constitute a generalization of branch groups. Abért's strategy is based on separating actions and uses a probabilistic argument that could be useful for other investigations or studies. Our proof for the case of branch pro-$p$ groups is shorter and more elementary than the one of Abért.

\subsection*{Acknowledgements} We would like to thank Rostislav Grigorchuk for suggesting the problem and for helpful discussions, and to Steffen Kionke and Matteo Vannaci for their useful comments. The first author would also like to thank Texas A\&M University for its warm hospitality while this work was being carried out.

\section{Preliminaries}
\label{section: Preliminaries}

\subsection*{\textit{\textmd{Notation}}} We shall write $H\le_f G$ to indicate $H$ is a finite-index subgroup of $G$. The action on the tree will be considered to be a left action, i.e. the composition $gh$ acts on $T$ from right to left. The order of an element $g\in G$ is denoted~$o(g)$. The cardinality of a set $S$ is denoted $\# S$ whereas the cardinality of a group $G$ is denoted~$|G|$.

\subsection{Faithful actions on rooted trees}
A \textit{spherically homogeneous rooted tree}~$T$ is a tree with a root $\emptyset$, where the vertices at the same distance from the root have all the same number of descendants. The vertices at distance exactly $n\ge 0$ from the root form the \textit{$n$th level of $T$}, denoted $\mathcal{L}_n$, and the ones at distance at most $n$ from the root form the \textit{$n$th truncated tree of $T$}, denoted $T^n$. For any vertex $v\in T$ the subtree rooted at $v$, which is again a spherically homogeneous rooted tree, is denoted $T_v$. The group of graph automorphisms of $T$ is denoted $\mathrm{Aut}~T$. Note that this definition is also valid if $T$ is replaced by any of the previously defined trees.

Given a vertex $v\in T$ we write $\mathrm{st}(v)\le \mathrm{Aut}~T$ for the stabilizer of the vertex $v$ and define $\mathrm{St}(n)\le \mathrm{Aut}~T$ via $\mathrm{St}(n)=\bigcap_{v\in \mathcal{L}_n}\mathrm{st}(v)$ for $n\ge 1$. Note that the $n$th level stabilizer $\mathrm{St}(n)$ is a normal subgroup of $\mathrm{Aut}~T$.

For any $1\le k\le \infty$ and $v\in T$ we define the section of depth $k$ at $v$ of an automorphism $g\in \mathrm{Aut}~T$ as the unique automorphism $g|_v^k$ of the truncated tree $T_v^k$ (for $k=\infty$ it is simply $T_v$) such that for any $w\in T_v^k$
$$g(vw)=g(v)g|_v^k(w).$$
Sections satisfy the following equality:
\begin{align}
\label{align: property of sections}
    (gh)|_v^k=(g|_{h(v)}^k)(h|_v^k).
\end{align}

For every $k\ge n\ge 1$ we define the isomorphism $$\psi_n^k:\mathrm{Aut}~T^k\to (\mathrm{Aut}~T_{v_1}^{k-n}\times\dotsb\times \mathrm{Aut}~T_{v_{N_n}}^{k-n})\rtimes \mathrm{Aut}~T^n$$
via
$$g\mapsto (g|_{v_1}^{k-n},\dotsc,g|_{v_{N_n}}^{k-n}) g|_\emptyset^n,$$
where $v_1,\dotsc,v_{N_n}$ are all the distinct vertices in $\mathcal{L}_n$.

Let us fix now a subgroup $G\le \mathrm{Aut}~T$. The group $G$ is said to be \textit{level-transitive} if $G$ acts transitively on $\mathcal{L}_n$ for all $n\ge 1$. We define $\mathrm{st}_G(v):=\mathrm{st}(v)\cap G$ and $\mathrm{St}_G(n):=\mathrm{St}(n)\cap G$ for any $v\in T$ and any $n\ge 1$. Furthermore, for $v\in T$, we define the \textit{rigid vertex stabilizer} $\mathrm{rist}_G(v)$ as the subgroup of $G$ consisting of automorphisms which fix $v$ and every vertex not in $T_v$. It is immediate from the definition that for $v\ne w$ at the same level of $T$ their rigid vertex stabilizers $\mathrm{rist}_G(v)$ and $\mathrm{rist}_G(w)$ have trivial intersection and commute with each other. Thus, for any $n\ge 1$ we may define the \textit{$n$th rigid level stabilizer} $\RiSt_G(n)\le G$ as the direct product 
$$\RiSt_G(n):=\prod_{v\in \mathcal{L}_n}\mathrm{rist}_G(v).$$
Note that if $G$ is level-transitive then the rigid vertex stabilizers of vertices at the same level are conjugate in $G$ and therefore 
\begin{align}
\label{align: rist is conjugate}
\RiSt_G(n)=\prod_{g\in G}\mathrm{rist}_G(v)^g.
\end{align}

A group $G\le \mathrm{Aut}~T$ is said to be \textit{branch} if $G$ is level-transitive and for every $n\ge 1$ the rigid level stabilizer $\RiSt_G(n)$ is of finite index in $G$; see \cite{Grigorchuk2000}. A group $G$ is said to be \textit{weakly branch} if it is level-transitive and $\RiSt_G(n) \neq 1$ for all $n \in \NN$.

For $S\subseteq \mathrm{Aut}~T$, $k \in \NN$ and $v \in T$, we define $\pi_k(S_v)\subseteq \mathrm{Aut}~T_v^k$ via
$$\pi_k(S_v):=\{s|_v^k~:~ s\in S\}.$$
If $v$ is taken to be the root we may drop the $v$ from the notation and simply write~$\pi_k(S)$. The following is a straightforward observation that will be useful in the proof of \cref{Theorem: torsion}:

\begin{lemma}
\label{lemma: properties of rists}
Let $G\le \mathrm{Aut}~T$. Then:
\begin{enumerate}[\normalfont(i)]
    \item for any $v\in \mathcal{L}_n$ we have $\pi_{k}(\RiSt_G(n)_v)=\pi_{k}(\mathrm{rist}_G(v)_v)$;
    \item for any $k\ge n\ge 1$ we have $\pi_k(\RiSt_G(n))=\prod_{w\in \mathcal{L}_n}\pi_{k-n}(\RiSt_G(n)_w)$;
    \item for any $k\ge n\ge 1$ and $v\in \mathcal{L}_n$ we have $|\pi_k(\mathrm{rist}_G(v))|=|\pi_{k-n}(\mathrm{rist}_G(v)_v)|$.
\end{enumerate}
\end{lemma}

\subsection{Pro-$p$ groups and the Haar measure}

The automorphism group $\Aut~T$ is isomorphic to the inverse limit $\varprojlim \pi_n(\Aut~T)$. Thus $\Aut~T$ is a profinite group with respect to the \textit{congruence topology}, i.e. the topology where the subgroups $\St_G(n)$ form an open basis of neighborhoods of the identity. In particular $\Aut~T$ is Hausdorff and compact as a topological group. If $G$ is a closed subgroup of $\Aut~T$, then $G$ is compact and we can associate to $G$ a normalized left-invariant Haar measure $\mu$ in the Borel $\sigma$-algebra of $G$, i.e., a measure $\mu$ such that $\mu(G) = 1$ and for any $g \in G$ and any $S$ measurable subset of $G$, we have $\mu(gS) = \mu(S)$. Moreover, for any measurable subset $S\subseteq G$ its Haar measure $\mu(S)$ is given by 
\begin{align}
\label{align: def haar measure}
\mu(S)=\lim_{k\to \infty}\frac{\#\pi_k(S)}{|\pi_k(G)|}.
\end{align}

A group $G$ is said to be a \textit{pro-$p$} group if it is isomorphic to an inverse limit of finite $p$-groups. A topological group is called \textit{topologically finitely generated} if it contains a dense finitely generated subgroup. A topologically finitely generated group is said to have \textit{finite rank} if every closed subgroup is topologically generated by at most $r$ elements for some $r\ge 1$. A topological group $G$ is \textit{$p$-adic analytic} if $G$ admits a $p$-adic analytic manifold structure such that the product and the inversion maps are analytic. These two concepts are linked by the following classical result:

\begin{proposition}[{see {\cite[Theorems 3.13 and 8.2]{padic}}}]
\label{proposition: finite rank}
For a finitely generated pro-$p$ group $G$ the following are equivalent:
\begin{enumerate}[\normalfont(i)]
    \item $G$ is $p$-adic analytic;
    \item $G$ has finite rank;
    \item $G$ has a powerful open subgroup.
\end{enumerate}
\end{proposition}

For our purposes it is enough to know that topologically finitely generated abelian pro-$p$ groups are powerful.

Finally, an infinite topological group is called \textit{just-infinite} if every proper quotient by a closed normal subgroup is finite.

\section{Proof of the main result}
\label{section: proof of the main result}

We need the following lemma, which states that a finite-index subgroup $H\le_f G$ has roughly the same number of sections of depth $k\ge 1$ as $G$ at a fixed vertex $v\in T$.

\begin{lemma}
\label{lemma_estimation_number_of_sections}
Let $H \le_f G \leq \mathrm{Aut}~T$, a natural number $n\ge 1$ and $v\in \mathcal{L}_n$. Then for all $k \ge 1$ we have
$$|\pi_k(G_v)| \leq \#\mathcal{L}_n  |G:H|   |\pi_k(H_v)|.$$
\end{lemma}
\begin{proof}
By assumption, there exists a finite set $S$ such that $G=\bigsqcup_{s\in S}sH$ and $\#S=|G:H|<\infty$.
Thus any $g\in G$ may be written as $g=sh$, where $s\in S$ and $h\in H$. Hence by \cref{align: property of sections}, for any $k\ge 1$ we get
$$g|_v^k=(sh)|_v^k=(s|_{h(v)}^k)(h|_v^k).$$
Therefore
\begin{align*}
    |\pi_k(G_v)|&\le \#\{s|_w^k~:~ s\in S\text{ and }w\in \mathcal{L}_n\}  |\pi_k(H_v)|\\
&\le \#\mathcal{L}_n   |G:H|  |\pi_k(H_v)|.\qedhere
\end{align*}
\end{proof}

Now we are in position to prove \cref{Theorem: torsion}, namely that the set of torsion elements in a branch pro-$p$ group has Haar measure zero.

\begin{proof}[Proof of \cref{Theorem: torsion}]
An element $g\in G$ has order $r\ge 1$ if and only if there exists a level $n\ge 1$ such that $g|_\emptyset^k\in \pi_k(G)$ has order $r$ for every $k\ge n$. Therefore the set~$F$ of torsion elements in $G$, admits the following decomposition:
\begin{align}
\label{align: first decomposition of S}
F&= \bigcup_{r=1}^\infty \set{g \in G~:~  o(g) = r}=\bigcup_{r=1}^\infty \bigcup_{n=1}^\infty \bigcap_{k=n}^\infty \set{g \in G~:~ o(g|_\emptyset^k) = o(g|_\emptyset^n)=r}. 
\end{align}
For $k\ge n\ge 1$ and $r\ge 1$, let us define the set $P_{r,n}(k)\subseteq \pi_k(G)$ via
$$P_{r,n}(k) = \set{h \in \pi_k(G)~:~  o(h) = o(h|^n_\emptyset) = r}.$$
Then, \cref{align: first decomposition of S} may be read as
\begin{align*}
F= \bigcup_{r = 1}^\infty \bigcup_{n = 1}^\infty \bigcap_{k = n}^\infty \pi_k^{-1}(P_{r,n}(k)).    
\end{align*}
Moreover, as $o(\pi_k(h))$ divides $o(\pi_{k+1}(h))$, we observe that 
$$\pi_{k+1}^{-1}(P_{r,n}(k+1)) \subseteq \pi_k^{-1}(P_{r,n}(k))$$
for every $k\ge n\ge 1$ and $r\ge 1$. Hence, by countable subadditivity and continuity from above of the measure $\mu$ we obtain
\begin{align*}
\mu(S) \leq \sum_{r = 1}^\infty \sum_{n = 1}^\infty \mu \left(\bigcap_{k = n}^\infty \pi_k^{-1}(P_{r,n}(k)) \right) = \sum_{r = 1}^\infty \sum_{n = 1}^\infty \lim_{k \rightarrow \infty} \mu \left( \pi_k^{-1}(P_{r,n}(k)) \right).
\end{align*}
Thus, it is enough to prove that 
$$\lim_{k \rightarrow \infty} \mu \left( \pi_k^{-1}(P_{r,n}(k)) \right)=0$$ for any $n,r \in \NN$.
By \cref{align: def haar measure} we have
\begin{align}
\label{align: haar measure of Pnk}
\mu \left( \pi_k^{-1}(P_{r,n}(k)) \right) = \frac{\# P_{r,n}(k)}{\abs{\pi_k(G)}}.
\end{align}

Our goal is to find an upper bound for \cref{align: haar measure of Pnk} in terms of $k$ such that it tends to 0 as $k$ tends to infinity. To this end let us first find an upper bound for the numerator $\# P_{r,n}(k)$.

Let $h \in P_{r,n}(k)$. Let us write $\tau:=h|^n_\emptyset \in \pi_n(G)$, which is of order $r$, and $h_v:=h|_v^{k-n} \in \pi_{k-n}(G_v)$ for every $v\in \mathcal{L}_n$ to simplify notation. As $h \in P_{r,n}(k)$, we have $h^r=1$ and for every vertex $v\in \mathcal{L}_n$ we get the equality
\begin{align}
\label{ec_sections_h^r}
1=h^r|_v^{k-n} =(h|_{\tau^{r-1}(v)}^{k-n})\dotsb (h|_{\tau(v)}^{k-n})(h|_v^{k-n}) = \prod_{i=1}^{r} h_{\tau^{r-i}(v)}.
\end{align}

The action of $\pi_n(G)$ on $\mathcal{L}_n$ is faithful so
$$r=o(\tau)=\mathrm{lcm}\{\# \mathrm{orb}_{\langle \tau\rangle}(v)~:~ v\in \mathcal{L}_n\}.$$
Moreover, since $\langle \tau\rangle \le \pi_n(G)$ is a finite $p$-group, all its orbits on $\mathcal{L}_n$ are of length a $p$-power. Hence the lcm of the lengths of the orbits is simply the maximum. Thus, there exists a vertex $v_0\in \mathcal{L}_n$ where this maximum is attained, i.e.
$$\#\mathrm{orb}_{\langle \tau\rangle}(v_0)=r.$$
This implies that the product in \cref{ec_sections_h^r} for $v=v_0$ is indexed by a set of $r$ vertices where each vertex appears exactly once. Therefore we may rearrange terms to obtain an expression of the form
$$h_{v_0}=h_{v_1}^{-1}\dotsb h_{v_{r-1}}^{-1},$$
where $v_1,\dotsc,v_{r-1}$ are $r-1$ distinct vertices in $\mathcal{L}_n\setminus\{v_0\}$. In other words $h_{v_0}$ is completely determined by $\{h_v\}_{v\in \mathcal{L}_n\setminus\{v_0\}}$. Thus applying the isomorphism $\psi_n^k$ to $P_{r,n}(k)$ yields the upper bound
\begin{align}
\label{align: upper bound}
\#P_{r,n}(k)=\#\psi_n^k(P_{r,n}(k)) \leq \abs{\pi_n(G)} \prod_{v \in 
\mathcal{L}_n \setminus \set{v_0}} \abs{\pi_{k-n}(G_v)}.
\end{align}

Now let us give a lower bound for the denominator in \cref{align: haar measure of Pnk}. Since $\RiSt_G(n)\le G$, we get
\begin{align}
\label{align: lower bound}
\abs{\pi_k(G)} \geq \abs{\pi_k(\RiSt_G(n))} = \prod_{v \in \mathcal{L}_n} \abs{\pi_{k-n}(\RiSt_G(n)_v)},
\end{align}
where the last equality is precisely \cref{lemma: properties of rists}\textcolor{teal}{(ii)}.

Plugging in \cref{align: haar measure of Pnk} both the bounds in \cref{align: upper bound} and in \cref{align: lower bound} yields
\begin{align*}
\mu(\pi_k^{-1}(P_{r,n}(k))) \leq \abs{\pi_n(G)} \Bigg( \prod_{v \in 
\mathcal{L}_n\setminus \set{v_0}}  \frac{ \abs{\pi_{k-n}(G_v)}}{\abs{\pi_{k-n}(\RiSt_G(n)_v)}} \Bigg) \abs{\pi_{k-n}(\RiSt_G(n)_{v_0})}^{-1}.
\end{align*}

Since $G$ is branch, the normal subgroup $\RiSt_G(n)$ is of finite index in $G$. Therefore we may apply \cref{lemma_estimation_number_of_sections} to obtain
\begin{align*}
\mu(\pi_k^{-1}(P_{r,n}(k)) &\leq \abs{\pi_n(G)} \Bigg( \prod_{v \in \mathcal{L}_n \setminus \set{v_0}} \#\mathcal{L}_n   |G:\RiSt_G(n)| \Bigg)  \abs{\pi_{k-n}(\RiSt_G(n)_{v_0})}^{-1}\\
&=\alpha(n)  \abs{\pi_{k-n}(\RiSt_G(n)_{v_0})}^{-1},
\end{align*}
where $\alpha(n)$ is a constant depending only on $n$.

Again, since $G$ is branch $\RiSt_G(n)$ is infinite. Thus, by \cref{align: rist is conjugate} $\mathrm{rist}_G(v_0)$ must be infinite as $G$ is level-transitive. Then we get
$$\lim_{k\to \infty}\abs{\pi_{k-n}(\RiSt_G(n)_{v_0})}=\lim_{k\to \infty}\abs{\pi_{k-n}(\mathrm{rist}_G(v_0)_{v_0})}=\lim_{k\to\infty}|\pi_k(\mathrm{rist}_G(v))|=\infty,$$
where the first and the second equalities follow from 
\cref{lemma: properties of rists}\textcolor{teal}{(i)} and \textcolor{teal}{(iii)} respectively and the third one from $\mathrm{rist}_G(v_0)$ being closed in $G$. Hence
\begin{align*}
\lim_{k\to \infty} \mu(\pi_k^{-1}(P_{r,n}(k))&=0.\qedhere
\end{align*}
\end{proof}

We finally prove \cref{Corollary: onto branch groups}. We recall the statement: If $G$ is a compact topological group admitting a branch pro-$p$ quotient, then the Haar measure of the set of its torsion elements has measure zero.

\begin{proof}[Proof of \cref{Corollary: onto branch groups}]
Let $H$ be a branch pro-$p$ group with Haar measure $\nu$ and $\varphi: G \rightarrow H$ a continuous epimorphism. Since $\varphi$ is continuous and surjective and the Haar measure is unique, then $\nu(A) = \mu(\varphi^{-1}(A))$ for all measurable sets $A \subseteq H$. In particular, if $F$ denotes the set of torsion elements of $G$ and $F'$ the set of torsion elements of $H$, we have $F \subseteq \varphi^{-1}(F')$ and consequently $$0 = \nu(F') = \mu(\varphi^{-1}(F')) \geq \mu(F),$$ and therefore $\mu(F) = 0$.
\end{proof}

\section{A $p$-adic analytic counterexample}
\label{sec_Example}

In this section, we show the existence of topologically finitely generated pro-$p$ groups whose sets of torsion elements have non-zero Haar measure. The following example in the case $p = 2$ was suggested by Grigorchuk who was informed by Zelmanov and Shumyatsky. Here, we extend this example to an odd prime number $p$ and realize it as a group acting faithfully and level-transitively on the $p$-adic tree.

\subsection{The construction of $G$}
Let $\xi$ be a primitive $p$-th root of unity, i.e. a zero of the polynomial $x^{p-1}+\dotsb+x+1$. Then for every $1\le y\le p-1$ the number $\xi^y$ is also a primitive $p$-th root of unity and thus
\begin{equation}
\sum_{k = 0}^{p-1} \xi^{yk} = 0.
\label{ec_p-root_of_unity}
\end{equation}

Denote with $\ZZ_p$ the $p$-adic integers and $A := \ZZ_p[\xi]$ the ring extension of $\ZZ_p$ with~$\xi$. Define $G: = A \rtimes_\varphi \mathbb{Z}/p\mathbb{Z}$, where $\varphi: \mathbb{Z}/p\mathbb{Z} \rightarrow \Aut(A)$ is given by $1 \mapsto (x \mapsto \xi x)$. 

The group $A$ is isomorphic as an additive group to the direct product
$$\ZZ_p \times \overset{p}{\dotsb}\times \ZZ_p$$ 
and therefore both $A$ and $G$ are topologically finitely generated pro-$p$ groups. Moreover since $A$ is a powerful open subgroup of $G$, then $G$ is of finite rank and $p$-adic analytic by \cref{proposition: finite rank}.

We shall use the notation $G=\{(x,y):x\in A, y\in \mathbb{Z}/p\mathbb{Z}\}$ for semidirect products. In this notation the product in $G$ is simply
$$(x,y) (z,w) = (x + \varphi_y(z), y+w). $$

We now prove that the set  $F$ of torsion elements of $G$ is given by
$$F=\{(0,0)\}\cup (G\setminus A)\subset G.$$

By induction on $n\ge 1$, we obtain
$$(x,y)^n = (x + x\xi^y + \dotsb + x\xi^{y(n-1)}, ny).$$

Then, if $y \neq 0$, we get
$$(x,y)^p = \left( x\sum_{k = 0}^{p-1} \xi^{yk}, py \right) = (0,0)$$
by \cref{ec_p-root_of_unity}. On the other hand, if $y = 0$, then
$$(x,y)^n = (nx, 0),$$
which is zero only if $x = 0$ as $\ZZ_p$ is torsion-free as an additive group. 

As $G$ is an infinite compact group, we have $\mu(\set{(0,0)}) = 0$ and since $|G: A| =p$, we have
$$\mu(F) =\mu\left(G\setminus A\right)=1-\mu(A)=1-\frac{1}{p}=\frac{p-1}{p}.$$

\subsection{Realization of $G$ in the $p$-adic tree}
The group $G$ can be realized in the $p$-adic tree $T$ as follows. Let $\sigma = (1\,\dotsb\,p)$ be the cyclic permutation in $\Sym(p)$, and define the elements $a,g,h \in \Aut~T$ as 
$$a = (1,\dotsc,1)\sigma,\quad g = (1,\dotsc,1,g)\sigma \text{\quad and\quad} h = (1,\dotsc,1,g).$$
Since $g^p = (g,\dotsc,g)$, then $o(g) = \infty$ and consequently $o(h) = \infty$. On the other hand, $o(a) = p$. Notice that if we denote $h_i = a^iha^{-i}$, then $h_i$ and $h_j$ commute for all $1 \leq i,j \leq p$ and so the subgroup $H := \<h>^{\langle a,h\rangle} = \<h_1,\dotsc,h_p> \lhd \<a,h>$ is isomorphic to $\ZZ\times\overset{p}{\dotsb}\times \ZZ$. Then $\<a,h>$ is isomorphic to $H\rtimes \langle a\rangle$ and thus the closure $\overline{\langle a,h\rangle}$ is isomorphic to~$G$.

Finally, since $G$ is of finite rank it cannot be branch (nor weakly branch). Indeed, if $\mathrm{RiSt}_G(n)\ne 1$ for every $n\ge 1$ then by \cref{align: rist is conjugate} the minimal number of generators of $\mathrm{RiSt}_G(n)$ is $p^n$ and thus $G$ cannot be of finite rank.
 
% End of document

% References

\bibliographystyle{unsrt}

\end{document}